\documentclass[11pt, amstex]{amsart}
\usepackage{calc,amssymb,amsmath,ulem}
\usepackage{alltt}
\RequirePackage[dvipsnames,usenames]{color}

\input{kmacros3.sty}
\input{xy}
\xyoption{all}
\usepackage{hyperref}
\usepackage{graphicx}
\usepackage[all,cmtip]{xy}
\usepackage{verbatim}
\usepackage[left=1.2in,top=1in,right=1.2in,bottom=1in]{geometry}

\theoremstyle{theorem}
\newtheorem*{mainthm*}{Main Theorem}

\renewcommand{\O}{\mathcal O}
\newcommand{\ie}{{\itshape i.e.} }
\newcommand{\roundup}[1]{\lceil #1 \rceil}

\renewcommand{\to}[1][]{\xrightarrow{\ #1\ }}

\newcommand{\Tt}{{\mathfrak{T}}}

\begin{document}
\numberwithin{equation}{theorem}
\title[Test ideals via a single alteration]{Test ideals via a single alteration and\\ discreteness and rationality of $F$-jumping numbers}
\author{Karl Schwede, Kevin Tucker, Wenliang Zhang}
\begin{abstract}
Suppose $(X, \Delta)$ is a log-$\bQ$-Gorenstein pair.  Recent work of M.~Blickle and the first two authors gives a uniform description of the multiplier ideal $\mJ(X;\Delta)$ (in characteristic zero) and the test ideal $\tau(X;\Delta)$ (in characteristic $p > 0$) via regular alterations.  While in general the alteration required depends heavily on $\Delta$, for a fixed Cartier divisor $D$ on $X$ it is straightforward to find a single alteration (\textit{e.g.} a log resolution) computing $\mJ(X; \Delta + \lambda D)$ for all $\lambda \geq 0$.  In this paper, we show the analogous statement in positive characteristic: there exists a single regular alteration computing $\tau(X; \Delta + \lambda D)$ for all $\lambda \geq 0$.  Along the way, we also prove the discreteness and rationality for the $F$-jumping numbers of $\tau(X; \Delta+ \lambda D)$ for $\lambda \geq 0$ where the index of $K_X + \Delta$ is arbitrary (and may be divisible by the characteristic).
\end{abstract}
\thanks{The first author was partially supported by the NSF grant DMS \#1064485.}
\thanks{The second author was partially supported by an NSF postdoctoral fellowship \#1004344.}
\thanks{The third author was partially supported by the NSF grant DMS \#1068946.}
\address{Department of Mathematics\\ The Pennsylvania State University\\ University Park, PA, 16802, USA}
\email{schwede@math.psu.edu}
\address{Department of Mathematics\\ Princeton University\\Princeton, NJ, 08544, USA}
\email{kftucker@princeton.edu}
\address{Department of Mathematics\\ University of Michigan\\ Ann Arbor, MI, 48109, USA}
\email{wlzhang@umich.edu}
\maketitle

\section{Introduction}

Suppose that $X$ is a normal algebraic variety over a perfect field $k$ and $(X, \Delta \geq 0)$ is an effective log-$\bQ$-Gorenstein pair.  In \cite[Theorems 4.6, 8.2]{BlickleSchwedeTuckerTestAlterations}, the authors showed
\[
\Image\Big( \pi_* \O_Y(\lceil K_Y - \pi^* (K_X + \Delta)  \rceil) \to[\Tr] \O_X \Big) = \left\{
  \begin{array}[c]{lc}
    \mJ(X;\Delta) &
    \begin{array}[t]{c}
      \mbox{ the multiplier ideal,} \\ \mbox{ if } \mathrm{char}(k) = 0
    \end{array}   \medskip \\
    \tau(X;\Delta) &  \begin{array}[t]{c}
      \mbox{ the test ideal,} \\ \mbox{ if } \mathrm{char}(k) = p > 0
    \end{array}
  \end{array}
\right.
\]
for a sufficiently large regular separable alteration $\pi : Y \to X$, where $\Tr : K(Y) \to K(X)$ is the corresponding trace map of function fields.
In general the alteration required depends heavily on $\Delta$; however, when working in characteristic zero, one can often perturb $\Delta$ while using the same alteration.  More precisely, fix a Cartier divisor $D \geq 0$ on $X$. Any regular alteration $\pi \: Y \to X$ such that
\[
\Supp(K_Y - \pi^* (K_X + \Delta)) \cup \Supp(\pi^{*}D)
\]
is a divisor with simple normal crossings (\textit{e.g.} a log resolution) will suffice to  compute $\mJ(X; \Delta + \lambda D)$ for all $\lambda \geq 0$.
Our principal result is to show the analogous statement for the test ideal in positive characteristic.

\begin{mainthm*}[Theorem \ref{thm.MainTheorem}]
Let $X$ be a normal algebraic variety over a perfect field of characteristic $p > 0$
and suppose $\Delta \geq 0$ is a $\bQ$-divisor on $X$ such that $K_X + \Delta$ is $\bQ$-Cartier.  If $D \geq 0$ is a Cartier divisor on $X$, there exists a finite separable cover (alternatively, a regular separable alteration) $\pi : Y \to X$ such that
\[
\tau(X; \Delta + \lambda D) = \Image\Big( \pi_* \O_Y(\lceil K_Y - \pi^*(K_X + \Delta + \lambda D) \rceil) \to[\Tr] K(X) \Big)
\]
for all real numbers $\lambda \geq 0$, where $\Tr : K(Y) \to K(X)$ is the corresponding trace map of function fields.
\end{mainthm*}
\noindent  In fact, in producing the finite separable cover above, one may take $X$ to be an arbitrary $F$-finite normal scheme (rather than just a variety over a perfect field).

Our method of proof relies, perhaps unsurprisingly, on the discreteness and rationality for the $F$-jumping numbers of $\tau(X; \Delta+ \lambda D)$ -- a result which is \textit{a posteriori} equivalent to the theorem above.  Nevertheless, in order to prove our main result, we must first generalize existing discreteness and rationality results to show the following.

\begin{theorem*}[Theorem \ref{thm.DiscAndRat}]
Suppose that $X$ is an $F$-finite normal integral scheme, $D$ is a Cartier divisor, and $\Delta$ is a $\bQ$-divisor such that $K_X + \Delta$ is $\bQ$-Cartier (with arbitrary index).  Then the $F$-jumping numbers of $\tau(X; \Delta+\lambda D)$ for $\lambda \geq 0$ are discrete and rational.
\end{theorem*}

\noindent
This result follows easily by combining the main result of \cite{BlickleSchwedeTakagiZhang} (which in turn builds upon \cite{KatzmanLyubeznikZhangOnDiscretenessAndRationality,BlickleMustataSmithFThresholdsOfHypersurfaces}) together with the main result of \cite{SchwedeTuckerTestIdealFiniteMaps}. As far as we are aware, our main theorem is also the first application of the discreteness and rationality of these $F$-jumping numbers.
\vskip 6pt
\noindent {\it Acknowledgements.}  The authors would like to thank the referee for several helpful comments.  
Furthermore, the authors would like to thank Manuel Blickle for valuable and inspiring conversations.  In fact, this paper was developed when all of the authors visited the Johannes Gutenberg-Universit\"at Mainz during the summer of 2011.  This visit was partially funded by the SFB/TRR45 \emph{Periods, moduli, and the arithmetic of algebraic varieties}.

\section{Preliminaries}
\label{sec.Preliminaries}

Throughout this paper, all schemes (and rings) are Noetherian, separated, have characteristic $p > 0$, and are always $F$-finite.
In particular, they are also excellent \cite{KunzOnNoetherianRingsOfCharP} and are assumed to possess a dualizing complex (\textit{cf.} \cite{Gabber.tStruc}).  In fact, very little is lost if one restricts to the study of varieties of finite type over a perfect field (as in the introduction).  For an integral scheme $X$, we use $\omega_X$ to denote the first non-zero cohomology of the dualizing complex $\omega_X^{\mydot}$.


On a normal integral scheme $X$, a \emph{canonical divisor} $K_{X}$ is any integral Weil divisor such that $\O_X(K_X) \cong \omega_X$. A $\bQ$-divisor (respectively $\bR$-divisor) $\Delta$ is a formal linear combination of prime Weil divisors with coefficients in $\bQ$ (respectively $\bR$).   Writing $\Delta = \sum a_i D_i$ where the $D_i$ are distinct prime divisors, we use $\roundup{\Delta} = \sum \roundup{a_i} D_i$ to denote the round-up of $\Delta$. We say a $\bQ$-divisor $\Delta$ is \emph{$\bQ$-Cartier} if there exists an integer $n > 0$ such that $n \Delta$ has integer coefficients and as such is also a Cartier divisor.  In this case, the smallest such $n$ is called the \emph{index} of $\Delta$.

\begin{definition}  Suppose $X$ is an integral scheme.
  \begin{enumerate}
  \item
We say $(X, \Delta)$ is a \emph{pair} if $X$ is normal and $\Delta$ is a $\bR$-divisor.  When additionally $K_X + \Delta$ is $\bQ$-Cartier, $(X, \Delta)$ is called \emph{log-$\bQ$-Gorenstein}.
\item
An \emph{alteration} $\pi : Y \to X$ is a generically finite proper dominant map from an integral scheme $Y$.  A \emph{finite cover} is an alteration which is also a finite map.
  \end{enumerate}
\end{definition}

For an alteration $f : Y \to X$, we have the Grothendieck trace map
\[
\Tr_{f} : f_* \omega_Y \to \omega_X
\]
coming from the theory of Grothendieck-Serre duality; for a detailed construction of this map, please see \cite[Section 2.1]{BlickleSchwedeTuckerTestAlterations}.
In the case that $f = F^e$ is the $e$-iterated Frobenius, we use $\Phi^e$ to denote $\Tr_{F^e}$ and call this map the \emph{canonical dual of Frobenius}.  In contrast, if $f$ is separable and $K_{Y} - \pi^{*}K_{X} = \Ram_{\pi}$ equals the ramification divisor wherever $\pi$ is finite, then the corresponding map $\Tr_{f} : \O_{Y}(K_{Y}) \to \O_{X}(K_{X})$ is induced by the trace map of function fields  $\Tr : K(Y) \to K(X)$.

As in \cite[Proposition 2.6]{BlickleSchwedeTuckerTestAlterations}, for any log-$\bQ$-Gorenstein pair $(X, \Delta)$ and any alteration $f : Y \to X$ from a normal scheme $Y$, the trace induces a natural map
\[
\Tr_{f} : f_* \O_Y(\lceil K_Y - f^*(K_X + \Delta) \rceil) \to K(X).
\]
In fact, the image of this map is contained in $\O_X(\lceil - \Delta \rceil)$, and so is inside $\O_{X}$ whenever $\Delta \geq 0$.  Using also that the trace is compatible with composition in general, this readily implies
\begin{equation}
\label{eq:factorsthrough}
\begin{array}{l}
\Image( g_* \O_{Y'}(\lceil K_{Y'} - g^*(K_X + \Delta) \rceil) \to[\Tr_{g}] K(X) ) \\ \qquad \qquad \qquad \qquad \quad \subseteq  \quad \Image( f_* \O_Y(\lceil K_Y - f^*(K_X + \Delta) \rceil) \to[\Tr_{f}] K(X) )
\end{array}
\end{equation}
whenever another alteration $g : Y' \to X$ from a normal scheme $Y'$ factors through $f$.

Rather than reviewing the definitions of the test ideals and parameter test modules of pairs (involving $\bR$-divisors) here,
we refer the reader to \cite[Definitions 2.23, 2.27]{BlickleSchwedeTuckerTestAlterations} in the $\bQ$-divisor case which generalizes to the settings of $\mathbb{R}$-divisors without change.  However, a few words of warning are in order.
First, whenever we speak of the test ideal throughout, we will always mean the big or non-finitistic test ideal; regardless, the classical or finitistic test ideal coincides with this notion in the log-$\bQ$-Gorenstein setting, \ie in (essentially) all situations considered in this paper.
Furthermore, we also caution that we allow the $\bR$-divisors~$\Delta$ in our pairs $(X, \Delta)$ to have \emph{negative coefficients}, as in \cite[Definition 6.12]{SchwedeTuckerTestIdealFiniteMaps} or \cite[Definition 2.27]{BlickleSchwedeTuckerTestAlterations}.  In particular, this means that the test ideal $\tau(X, \Delta)$ may not be an honest ideal sheaf, but rather it is simply a fractional ideal sheaf (respectively, the parameter test module $\tau(\omega_{X};\Gamma)$ may not be a submodule of $\omega_{X}$).
We briefly record the following list of relevant properties of test ideals and parameter test modules for pairs below.

\begin{proposition}
\label{prop.BasicPropertiesOfTau}
Suppose that $(X, \Delta)$ is a log-\mbox{$\bQ$-Gorenstein} pair, $\Gamma$ is a $\bQ$-Cartier $\bQ$-divisor, $t \geq 0$ is a real number, and $D \geq 0$ is a Cartier divisor.  Then
\begin{itemize}
\item[(1)]  \parbox{.9\textwidth}{\centering{$\tau(X; \Delta + D) = \tau(X; \Delta) \tensor \O_X(-D)$ and $\tau(\omega_X; \Gamma + D) = \tau(\omega_X; \Gamma) \tensor \O_X(-D)$.}} \\\cite[Page 9, Basic Property (ii)]{TakagiInterpretationOfMultiplierIdeals} and \cf \cite[Lemma 2.29]{BlickleSchwedeTuckerTestAlterations}.
\item[(2)]  For any sufficiently small real number $\varepsilon > 0$, \newline
\parbox{.9\textwidth}{\centering{ $\tau(X; \Delta + \varepsilon D) = \tau(X; \Delta)$ and $\tau(\omega_X; \Gamma + \varepsilon D) = \tau(\omega_X; \Gamma)$.}}
\newline
\cite[Lemma 3.23]{BlickleSchwedeTakagiZhang}.
\item[(3)]  For any finite cover $\pi : Y \to X$ with trace map $\Tr_{\pi}$ as above, we have \newline \parbox{.9\textwidth}{\centering $\Tr_\pi\left( \pi_* \tau(\omega_Y; \pi^* ( \Gamma + tD) ) \right) = \tau(\omega_X; \Gamma + tD)$.}  \newline \cite[Proposition 4.4]{BlickleSchwedeTuckerTestAlterations} \cf \cite{SchwedeTuckerTestIdealFiniteMaps}.
\item[(4)]  In particular, if $\Phi^e$ is the canonical dual of Frobenius, then
\newline
\parbox{.9\textwidth}{\centering $\Phi^{e}\big(F^e_* \tau(\omega_X; p^{e}\Gamma + t p^e D)\big) = \tau(\omega_X; \Gamma + t D)$.}
\item[(5)]  If $t' > t$, then
\newline
\parbox{.9\textwidth}{\centering $\tau(X; \Delta + t' D) \subseteq \tau(X; \Delta + tD)$.}  \newline \cite[Page 9, Basic Property (i)]{TakagiInterpretationOfMultiplierIdeals}.
\end{itemize}
\end{proposition}


\section{Discreteness and rationality}

\begin{definition}
Suppose that $X$ is a normal integral scheme, $\Delta$ and $\Gamma$ are $\bQ$-divisors, and $D$ is a Cartier divisor.
A real number $t \geq 0$ is an \emph{$F$-jumping number of the test ideals $\tau(X; \Delta + \lambda D)$ for $\lambda \geq 0$} if for all $\varepsilon > 0$, we have
\[
\tau(X; \Delta + (t - \varepsilon)D) \neq \tau(X; \Delta + {t}D).
\]
Similarly, a real number $t \geq 0$ is an \emph{$F$-jumping number of the parameter test modules $\tau(\omega_X; \Gamma + \lambda D)$ for $\lambda \geq 0$} if for all $\varepsilon > 0$ we have
\[
\tau(\omega_X; \Gamma + (t - \varepsilon)D) \neq \tau(\omega_X; \Gamma+ {t}D).
\]
Furthermore, we say the the $F$-jumping numbers of $\tau(X, \Delta + \lambda D)$ (respectively $\tau(\omega_X; \Gamma + \lambda D)$) are \emph{discrete and rational} if they are a set of rational numbers without any accumulation points (technically, this is a stronger condition than being discrete).
\end{definition}

\begin{lemma}
  \label{lem:idealmodulediscretesame}
Suppose $D \geq 0$ is a Cartier divisor on a normal integral scheme $X$, and let $m$ be a positive integer.  The following statements are equivalent:
\begin{enumerate}
\item For all $\bQ$-divisors $\Delta$ where $K_{X} + \Delta$ is $\bQ$-Cartier with index $m$, the $F$-jumping numbers of $\tau(X; \Delta + \lambda D)$ for $\lambda \geq 0$ are discrete and rational.
\item For all $\bQ$-Cartier $\bQ$-divisors $\Gamma$ with index $m$, the $F$-jumping numbers of $\tau(\omega_{X}; \Gamma + \lambda D)$ for $\lambda \geq 0$ are discrete and rational.
\end{enumerate}
In particular, \cite[Main Theorem]{BlickleSchwedeTakagiZhang} implies both statements for $m$ not divisible by $p$.
\end{lemma}

\begin{proof}
Setting $\Gamma = K_{X} + \Delta$, it follows from the argument of \cite[Lemma 2.29]{BlickleSchwedeTuckerTestAlterations} that we have
\[
\tau(\omega_{X}; \Gamma + \lambda D) = \tau(X; \Delta + \lambda D).
\]
for all real numbers $\lambda \geq 0$.  In particular, the corresponding $F$-jumping numbers are the same, and the desired equivalence follows immediately.  For the remaining assertion, as these statements are local, it is harmless to assume that $X$ is affine and $\Delta \geq 0$. When $p$ does not divide $m$, (1) is precisely the statement of \cite[Theorem 5.6]{BlickleSchwedeTakagiZhang}
\end{proof}

In order to prove our main result in the next section, we must first generalize the existing discreteness and rationality results to show the following.

\begin{theorem}
\label{thm.DiscAndRat}
Suppose $(X,\Delta)$ is a log-$\bQ$-Gorenstein pair and $D \geq 0$ is a Cartier divisor.
Then the $F$-jumping numbers of $\tau(X; \Delta + \lambda D)$ for $\lambda \geq 0$ are discrete and rational.
\end{theorem}
\begin{proof}
By Lemma~\ref{lem:idealmodulediscretesame}, it suffices to check that
the $F$-jumping numbers of $\tau(\omega_{X}; \Gamma + \lambda D)$ for $\lambda \geq 0$ are discrete and rational for any $\bQ$-Cartier $\bQ$-divisor $\Gamma$.
This statement is known if $p$ does not divide the index of $\Gamma$, and one can use Frobenius to reduce to this case in general.  Choose $e > 0$ such that $p^e \Gamma$ has index not divisible by $p$.  If $\Phi^{e}$ is the canonical dual of the $e$-iterated Frobenius $F^e$, we know that
\[
\Phi^e\big( F^e_* \tau(\omega_X; p^e \Gamma + \lambda p^e D )\big) = \tau(\omega_X; \Gamma + \lambda D )
\]
for any real number $\lambda \geq 0$ by Proposition \ref{prop.BasicPropertiesOfTau}(4).   Since the main result of \cite{BlickleSchwedeTakagiZhang} gives that the $F$-jumping numbers of
 $\tau(\omega_X; p^e \Gamma + \lambda D )$ for $\lambda \geq 0$ are discrete and rational, those of $\tau(\omega_X; \Gamma + \lambda D )$ must be as well.
\end{proof}

\begin{remark}
In the case that $X$ is essentially of finite type over an $F$-finite field and $\ba \subseteq \O_X$ is a non-zero ideal, a similar argument also yields the discreteness and rationality of the $F$-jumping numbers of $\tau(X; \Delta, \ba^{\lambda})$ for $\lambda \geq 0$ and any $\Delta$ such that $K_X + \Delta$ is $\bQ$-Cartier.  The simplest way to do this is to generalize Proposition \ref{prop.BasicPropertiesOfTau}(4) to this setting and proceed as above. Alternately, one can apply \cite[Theorem 6.23]{SchwedeTuckerTestIdealFiniteMaps} directly as follows.  We first construct $\Tt : F^e_* \O_X \to \O_X$ appearing in the notation of \cite{SchwedeTuckerTestIdealFiniteMaps}.  By working locally, we may assume that $K_X \leq 0$, and twisting $\Phi^e : F^e_* \omega_X \to \omega_X$ by $\O_X(-K_X)$ yields the composition
\[
\Tt : F^e_* \O_X \to[\subseteq] F^e_* \O_X(K_X - p^e K_X) \to[\Phi^{e} \otimes \O_{X}(-K_{X})] \O_X .
\]
Note that $\Tt$ corresponds to the divisor
\[
\text{Ram}_{\Tt} :=  K_X - p^e K_X
\]
as in \cite[Equation (2.2.1)]{SchwedeTuckerTestIdealFiniteMaps}, so that we have $\Delta_{X/\Tt} = p^e \Delta - \text{Ram}_{\Tt} = p^e (K_X + \Delta) - K_X$ in the notation of \cite[Setting 6.17]{SchwedeTuckerTestIdealFiniteMaps}.  Thus, for $e$ sufficiently large, $K_X + \Delta_{X/\Tt}$ is $\bQ$-Cartier with index not divisible by $p > 0$.  Therefore, the main theorem of \cite{BlickleSchwedeTakagiZhang} applies and we know that the $F$-jumping numbers of $\tau(X; \Delta_{X/\Tt}, (\ba^{[p^e]})^{\lambda})$ are discrete and rational.  Finally, by \cite[Theorem 6.23]{SchwedeTuckerTestIdealFiniteMaps}, so are the $F$-jumping numbers of $\tau(X; \Delta, \ba^{\lambda})$.
\end{remark}

\section{The existence of a single alteration}

In this section, we prove our main theorem. The crux of the argument lies in the following lemma.

\begin{lemma}
\label{lem.TauOmegaBound}
Suppose that $(X, \Delta)$ is a log $\bQ$-Gorenstein pair and that $D \geq 0$ is a Cartier divisor. Given two adjacent $F$-jumping numbers $t_0 < t_1$ of $\tau(X; \Delta + \lambda D)$ for $\lambda \geq 0$, there exists a finite separable cover $f : Y \to X$ such that
\[
\tau(X; \Delta + t D) ) = \Image\Big( f_* \O_Y(\lceil K_Y - f^*(K_{X}+ \Delta + t D) \rceil) \xrightarrow{\Tr_f} K(X) \Big)
\]
for all $t$ satisfying $t_0 \leq t < t_1$.
\end{lemma}
\begin{proof}
Note that $t_0$ is rational by Theorem \ref{thm.DiscAndRat}.  Using the main result, Theorem 4.6, of \cite{BlickleSchwedeTuckerTestAlterations}, we can find a finite separable cover $f : Y \to X$ such that $f^* (K_{X}+ \Delta + t_0 D)$ is Cartier and
\[
\tau(X; \Delta + t_0 D) = \Image\Big( f_* \O_{Y}(\lceil K_{Y} - f^*(K_{X} + \Delta + t_0 D) \rceil) \xrightarrow{\Tr_f} K(X)\Big).
\]
We need to prove that the same cover works for every real number $t$ satisfying $t_0 \leq t < t_1$.

Choose rational number $t'$ such that $t_0 \leq t \leq t' < t_1$.  Applying \cite[Theorem 4.6]{BlickleSchwedeTuckerTestAlterations}  once again, we can find another cover $g : Y' \to X$, factoring through $f$, such that $g^* (K_{X}+ \Delta + t'D)$ is Cartier and
\[
\tau(X; \Delta + t' D) = \Image\Big( g_* \O_{Y'}(\lceil K_{Y'} - g^*(K_{X}+ \Delta + t'D) \rceil) \xrightarrow{\Tr_g} K(X)\Big).
\]
Since $t_0<t_1$ are adjacent $F$-jumping numbers and $t,t'$ are between $t_0$ and $t_1$, we have
\[ \tau(X; \Delta + t_0 D)= \tau(X; \Delta + t D)= \tau(X; \Delta + t' D).\]
Now, we simply observe that
\[
\begin{array}{rl}
  & \tau(X; \Delta + t D)\\
=  &         \tau(X; \Delta + t_0 D)\\
= &         \Image\Big( f_* \O_{Y}(\lceil K_{Y} -  f^*(K_{X}+ \Delta + t_0 D) \rceil) \xrightarrow{\Tr_f} K(X)\Big)\\
\supseteq & \Image\Big( f_* \O_{Y}(\lceil K_{Y} - f^*(K_{X}+ \Delta +t D) \rceil) \xrightarrow{\Tr_f} K(X)\Big)\\
\supseteq & \Image\Big( f_* \O_{Y}(\lceil K_{Y} - f^*(K_{X}+ \Delta +t' D) \rceil) \xrightarrow{\Tr_f} K(X)\Big)\\
\supseteq & \Image\Big( g_* \O_{Y'}(\lceil K_{Y'} - g^*(K_{X}+ \Delta + t'D) \rceil) \xrightarrow{\Tr_g} K(X)\Big)\\
= & \tau(X; \Delta + t' D)\\
= & \tau(X; \Delta + t D),
\end{array}
\]
where the first two containments $\supseteq$ are valid simply because $t' > t > t_0$ and the last occurs because $g$ factors through $f$ as in \eqref{eq:factorsthrough}. It is evident that each containment must in fact be an equality, and so
\[
\tau(X; \Delta + tD) = \Image\Big( f_* \O_{Y}(\lceil K_{Y} - f^*(K_{X}+ \Delta + tD) \rceil) \xrightarrow{\Tr_f} K(X)\Big)
\]
as claimed.
\end{proof}

We now prove our main result.

\begin{theorem}
\label{thm.MainTheorem}
Suppose that $(X, \Delta)$ is a log $\bQ$-Gorenstein pair and $D \geq 0$ is a Cartier divisor.  Then there exists a single finite separable cover $\pi : Y \to X$ such that
\begin{equation}
\label{eq:wantinproof}
\tau(X; \Delta + \lambda D) = \Image\Big( \pi_* \O_Y(\lceil K_Y - \pi^*(K_X + \Delta + \lambda D) \rceil) \xrightarrow{\Tr_{\pi}} K(X) \Big)
\end{equation}
for all real numbers $\lambda \geq 0$.  Alternatively, if $X$ is of finite type over a $F$-finite (respectively perfect) field, one may take $\pi : Y \to X$ to be a regular (separable) alteration.
\end{theorem}
\begin{proof}
There are only finitely many adjacent pairs of $F$-jumping numbers of $\tau(X; \Delta + \lambda D)$ that are between zero and one by Theorem \ref{thm.DiscAndRat}, and
so we can find a single finite separable cover $\pi : Y \to X$ dominating all of the covers produced by Lemma \ref{lem.TauOmegaBound} for each individual pair.  Thus, we have that \eqref{eq:wantinproof} is valid
for all $\lambda \in [0, 1)$; using the projection formula and Proposition \ref{prop.BasicPropertiesOfTau}(1) gives that it is valid for all $\lambda \geq 0$ as desired.
The remaining statements when $X$ is of finite type over an $F$-finite or perfect field follow immediately from the arguments of \cite[Corollary 4.7]{BlickleSchwedeTuckerTestAlterations} utilizing \cite[Theorem 4.1]{de_jong_smoothness_1996}.
\end{proof}

\bibliographystyle{skalpha}
\bibliography{CommonBib}
\end{document}